\documentclass[12pt, a4paper, reqno]{amsart}

\usepackage{amsthm, mathtools, amssymb, amsfonts, stmaryrd, latexsym, mathrsfs, todonotes}
\usepackage{graphicx}
\usepackage[english]{babel}
\usepackage[all]{xy}
\usepackage{nicefrac}
\usepackage{mparhack}

\usepackage{color}

\usepackage{fancyhdr}
\usepackage{enumitem}
\usepackage{algorithmic}
\usepackage{algorithm}
\usepackage{hyperref}

\theoremstyle{plain}
\newtheorem{theo}{Theorem}[section]
\newtheorem{lemma}[theo]{Lemma}
\newtheorem{prop}[theo]{Proposition}
\newtheorem{corollary}[theo]{Corollary}
\newtheorem{remark}[theo]{Remark}

\theoremstyle{definition}

\newcommand{\F}{\mathbb{F}}

\newcommand{\Z}{\mathbb{Z}}
\newcommand{\Q}{\mathbb{Q}}

\newcommand{\NN}{\mathbb{N}}

\newcommand{\cK}{\mathcal{K}}
\newcommand{\cG}{\mathcal{G}}

\newcommand{\kk}{\kappa}

\renewcommand{\epsilon}{\varepsilon}
\newcommand{\I}[1]{\llbracket #1\rrbracket}

\newcommand{\Gal}{\mathop{\mathrm{Gal}}\nolimits}
\newcommand{\Aut}{\mathop{\mathrm{Aut}}\nolimits}

\newcommand{\ord}{\mathop{\mathrm{ord}}\nolimits}
\newcommand{\lcm}{\mathop{\mathrm{lcm}}\nolimits}
\newcommand{\Ind}{\mathop{\mathrm{Ind}}\nolimits}
\newcommand{\Core}{\mathop{\mathrm{Core}}\nolimits}

\newcommand{\ur}{\mathrm{ur}}

\newcommand{\sfr}[2]{\nicefrac{#1}{#2}}

\newcounter{enumi_saved}

\makeatletter
\def\imod#1{\allowbreak\mkern10mu({\operator@font mod}\,\,#1)}
\makeatother

\allowdisplaybreaks

\begin{document}

\title{On wild extensions of a $p$-adic field }
\author{Ilaria Del~Corso}
\author{Roberto Dvornicich}
\author{Maurizio Monge}
\email{delcorso@dm.unipi.it}

\email{dvornic@dm.unipi.it}

\email{maurizio.monge@gmail.com}
\date{\today}

\keywords{$p$-adic fields, Kummer theory,
  ramification theory}
\subjclass[2010]{11S05, 11S15}

\begin{abstract}

In this paper we consider the problem of classifying the isomorphism classes of extensions of degree $p^k$ of  a $p$-adic field, restricting to the case of extensions without intermediate fields.

We establish a correspondence between the isomorphism classes of these extensions and some Kummer extensions of a suitable field $F$ containing $K$. We then describe such classes in terms of the representations of
$\Gal(F/K)$.
 
  Finally, for $k=2$ and for each possible Galois group $G$, we count the number of isomorphism classes of the extensions whose normal closure has a Galois group isomorphic to $G$. As a byproduct, we get the total number of isomorphism classes.
\end{abstract}

\maketitle
\section{Introduction}
Let $p$ be a prime number, and let $K$ be a $p$-adic field.
A natural problem is to describe the set ${\mathcal E}_K(e,f)$ of all extensions of $K$ with fixed ramification index $e$ and inertial degree $f$.
In this setting one may consider several questions, such as
%For each pair $(e,f)$ of positive integers, let  ${\mathcal E}_K(e,f)$ be the set 
%of extensions of $K$ with ramification index $e$ and inertial degree
%$f$ is finite.
%% and denote by ${\mathcal I}_K(e,f)$ the set of its isomorphism classes. 
%
%In this context many natural problems arise, such as 
finding enumerating formulas and, more generally, characterizing subfamilies  of ${\mathcal E}_K(e,f)$  whose fields share some property: for example, the valuation of the discriminant, the Galois group of the normal closure, or, more finely, the full ramification filtration.  
These questions are elementary in the tame case ($p\nmid e$) and definitely more complicated in the wild case ($p|e$) and have been considered, over the years, by many authors.

%Let $p$ be a prime number, and let $K$ be a $p$-adic field.
%A natural problem is to describe the set ${\mathcal E}_K(e,f)$ of all extensions of $K$ with fixed ramification index $e$ and inertial degree $f$.
%The situation is elementary and easy to describe in the tame case ($p\nmid e$) and definitely more complicated in the wild case ($p|e$). In general, one may consider several questions such as
%%For each pair $(e,f)$ of positive integers, let  ${\mathcal E}_K(e,f)$ be the set 
%%of extensions of $K$ with ramification index $e$ and inertial degree
%%$f$ is finite.
%%% and denote by ${\mathcal I}_K(e,f)$ the set of its isomorphism classes. 
%%
%%In this context many natural problems arise, such as 
%finding enumerating formulas and, more generally, characterizing subfamilies  of ${\mathcal E}_K(e,f)$  whose fields share some property: for example, the valuation of the discriminant, the Galois group of the normal closure, or, more finely, the full ramification filtration.   
%Some of these problems have been considered, over the years, by many authors. 

%related to the characterisations of subset of  ${\mathcal E}_K(e,f)$ and to the determinination of explicit formulas for classes of  elements of this set.\smallskip
In 1962 Krasner
\cite{krasner1962}, by examining all possible Eisenstein polynomials, obtained  an explicit formula for
the total number of elements of ${\mathcal E}_K(e,f)$.  Later on,
Serre \cite{serre1978formule}, with a similar method, was able to refine
Krasner's result by counting the elements in ${\mathcal E}_K(e,f)$
with given discriminant for all possible discriminants; in the
same paper, he also showed that the distribution of the extensions
in ${\mathcal
  E}_K(e,f)$ according to their discriminant satisfies a ``mass
formula'' of general type.

\smallskip

%The set of the extensions of $\Q_p$ has been studied also by a computational approach by  Jones and Roberts \cite{jones2006database}, who produced a database which contains a rich description  of quite a number of extensions. However, the complexity of the computation grows very quickly with the degree $n$ of the extensions, and in particular with the power of $p$ dividing $n$; for example, the biggest prime for which all  extensions of degree $p^2$ of $\Q_p$ are classified is $p=3$.
% 
 
In 2004, Hou and Keating \cite{hou2004enumeration} considered the problem of determining the number of isomorphism classes of fields in
${\mathcal E}_K(e,f)$; they
found general
formulas when $p^2 \nmid e$ and, under some additional assumptions
on $e$ and $f$, also when $p^2|| e$. In 2011 one of the authors  was able to deduce a formula enumerating the isomorphism classes of extensions
of a $p$-adic field $K$ with given ramification $e$ and inertia $f$, for each $e$ and $f$ \cite{monge2010}.

In the paper \cite{delcorso2007compositum}, two of the authors present a complete description of the smallest degree case for which there is wild ramification, namely the case
 of extensions of degree $p$ of a $p$-adic field $K$.  
They introduced  a suitable extension $F$ of $K$  
%such that the degree $p$ extensions of $K$ are in one to one correspondence with a subset of the degree $p$ extensions %of $F$ which can be described  in terms of Galois representations (see Theorem \ref{CP2:theo1}). 
over which all  extensions of degree $p$ of $K$ become "special" Kummer extensions, in the sense that  they can be recognized, among all degree $p$ extensions of $F$, in terms of Galois representations (see Theorem \ref{CP2:theo1}). This method allows them to get a very detailed information on each field of the family; in particular, they derive 
the  formula for the number of the extensions of $K$ of degree $p$
and of their isomorphism classes according to their
discriminant. 
In 2010, Dalawat \cite{dalawat2010serre} revisited  the method of  \cite{delcorso2007compositum}, introducing a different language, and  extended the main result to  local fields of characteristic $p$. 

In this paper we consider the problem of classifying the isomorphism classes of the extensions of degree $p^k$ of a $p$-adic field for all $k\ge1$. 
We focus on the case of extensions with no intermediate field, since in this case we are able to adapt the same scheme of proof that was used for extensions of degree $p$ in \cite{delcorso2007compositum} and \cite{dalawat2010serre}.  An extension of Theorem \ref{CP2:theo1} and the study of the Galois representations allows us to obtain a classification of the extensions of degree $p^k$ in this case. Moreover, for $k=2$ we perform more explicit computations and we find formulas for the number of extensions whose normal closure has a given Galois group.

In the following we give an account of the methods and the results of this paper. The general case turns out to be substantially more involved then the degree $p$ case; however,  since the sequence of the main steps is the same, 
 to describe the present paper we find it useful to give first a sketch of  the proof of the case $k=1$; for details one can refer to 
\cite{delcorso2007compositum} and \cite{dalawat2010serre}.

\subsection*{The basic idea}
 An extension
$L/K$ of degree $p$ is either unramified and hence cyclic, or totally
ramified.
In the latter case, it is well known that
the Galois group of the normal closure $\tilde{L}$ of $L$ is isomorphic to $S\rtimes_\varphi H$, where $S\cong{}\Z/p\Z$ is 
the $p$-Sylow subproup and $H$ is the subgroup fixing $L$; moreover, 
the map
$\varphi\ :\ H\rightarrow\Aut(S)\cong{}\Z/p\Z^\times$ is injective and
$H$ is cyclic of order $d$ dividing $p-1$.
It is easy to see that each subgroup of index $p$ in $S\rtimes_\varphi H$ is generated by one
element of order $d$, and these elements are exactly the generators of
the conjugates of $H$. Hence $\tilde{L}$ can only be obtained
as the normal closure of extensions of degree $p$ that are in the same
isomorphism class.

Let $F$ be the compositum of all extensions of $K$ of exponent diving
$p-1$; note that $F$ contains the $p$-th roots of unity. If $L/K$ is an arbitrary extension of degree $p$, then the
composite extension $L_F=FL$ is clearly Galois over $K$. The group $\Gal (L_F/K)$ has a $p$-Sylow subgroup  $\hat{S}=\Gal(L_F/F)$ which
is normal and cyclic, and
$\hat{H}=\Gal(L_F/L)\cong{}\Gal(F/K)$ is a complement acting on
$\hat{S}$ by conjugation.
The subgroup  of $\hat{H}$ fixing $\tilde{L}$ is the maximal subgroup of $\hat{H}$ which is normal in $\Gal(L_F/K)\cong \hat{S}\rtimes\hat{H}$, so it coincides with the subgroup of $\hat{H}$ acting trivially on
$\hat{S}$. It follows that  from $L_F$ one can recover $\tilde{L}$ as the field fixed by the central
elements of $\Gal(L_F/K)$ of order prime to $p$.

A key point is that the set of fields $L_F$ obtained in this way  coincides with the set $\Omega$ of all extensions of degree $p$ of $F$ that are Galois over $K$ (see \cite[Lemma 4]{delcorso2007compositum}).  
By Kummer theory, the extensions of degree $p$ of $F$ correspond to  the subgroups of order $p$ of ${F^\times}/{(F^\times)^p}$; moreover, the elements of $\Omega$ correspond to those subgroups  that
are invariant under the action of $\Gal(F/K)$. 

Let $\rho:\,\Gal(F/K)\to{\rm Aut} ({F^\times}/{(F^\times)^p})$ be the representation associated to conjugation. Then the elements of  $\Omega$ correspond to the
irreducible representations contained in $\rho$ (which have all degree 1).
 
 Summarizing, we obtain the following theorem.

\begin{theo}
\label{CP2:theo1}
Let $K$ be a $p$-adic field, and $F$ be the compositum of all cyclic
extensions of exponent $p-1$. Then we have a natural one-to-one
correspondence between the isomorphism classes of extensions of degree
$p$ of $K$ and the irreducible subspaces of the $\Gal(F/K)$-module ${F^\times}/{(F^\times)^p}$.
\end{theo}

To enumerate the isomorphism classes it remains to study the Galois module structure of ${F^\times}/{(F^\times)^p}$, and in particular its irreducible subspaces and their multiplicity.

Finally, we remark that in this case the correspondence can be made very explicit, and makes it possible to recover easily
additional invariants of the extensions.

\bigskip
To generalize the result to the higher dimensional case, the first task is to find a suitable field $F$ for which an analogue of Theorem \ref{CP2:theo1} holds.
For each degree $p^k$, with $k\ge2$, restricting to the case of fields with no intermediate extension, we are able to  prove that in fact there are many possible choices for $F$ (see Section \ref{correspondence}) .

In Section \ref{section4} we describe the $\Gal(F/K)$-module structure of $F^\times/F^{\times p}$ for any tame and split extension $F$ of $K$. 
We start with the classification of the irreducible representations of $\Gal(F/K)$ over $\F_p$, via the study of the irreducible representations over $\bar \F_p$. Next, we determine the decomposition of $F^\times/F^{\times p}$ as a $\bar\F_p[\Gal(F/K)]$-module and then we bring back this information to get the structure of $F^\times/F^{\times p}$ as a $\F_p[\Gal(F/K)]$-module. Again, the results of this section hold  for every $k$.

In Section \ref{section5} we focus on $k=2$:
for each possible Galois group $G$, we count the number of isomorphism classes of the extensions whose normal closure has a Galois group isomorphic to $G$. As a byproduct, we get the total number of isomorphism classes and the total number of extensions of degree $p^2$ without intermediate extensions. 

We remark that, in principle, the results of Section \ref{section4} would allow to perform the same classification for any value of $k$. However, the number of cases to be examined increases with $k$ and therefore we restricted our analysis to the case $k=2$. Recently, M. R. Pati \cite{patik=l}, following our method, extended the classification to the case when $k$ is any prime number.

Finally, we observe that restricting to the case of extensions without intermediate fields,  has made it possible to find a tame  extension $F$ of $K$ for which  the correspondence theorem holds; in turn this yields that the  Galois module structure of $F^\times/F^{\times p}$ has a very nice form.  
One important feature is that all representations considered are induced by representations of groups of order coprime to $p$.   Trying to adapt the same scheme of proof to the general case, one can not hope for a tame extension $F$ of $K$ playing the same role, and so one has to consider representations of groups of order divisible by $p$.

In the case $k=2$ we can further subdivide the extensions not considered here, into two cases: those with just one intermediate field and those with more than one. The last case can be easily dealt using the results of \cite{delcorso2007compositum}, since each extension is the compositum of two extensions of degree $p$.
The case  of extensions with just one intermediate field requires different methods and will be the object of a separate paper by one of the authors.

\section{Notation and preliminaries}
\label{sec:2}

Throughout the paper we shall use the following notation. For a $p$-adic field $K$ we denote by $e_K$ and $f_K$ the absolute ramification index and inertial degree of $K$, respectively, by $\pi_K$ an uniformizing element, by $\kk_K$ its residue field and we put $q_K=|\kk_K|$.

If $E/K$ is a finite extension we denote by $e(E/K)$ and $f(E/K)$ the ramification index and the inertial degree of the extension; if $E/K$ is Galois with ${\rm Gal}(E/K)=G$, then $G=G_{-1}\supseteq G_0\supseteq G_1\supseteq\ldots$ is, as usual, the lower numbering ramification filtration. In particular, $G_0$ is the inertia group  and its fixed field $E^{\ur}$ is the maximal unramified subextension of $E/K$. 

We now recall some basic facts about the structure of the Galois group of a tamely ramified
normal extension $E/K$ (see for instance \cite{iwasawa1955galois}). Let $G=\Gal(E/K)$, then $G_0=\langle\tau\rangle$ is a cyclic group and
can be canonically embedded into $\kk_E^\times$ via the
map $\sigma\mapsto\overline{\sfr{\sigma(\pi_E)}{\pi_E}}$, which is
independent of the uniformizer $\pi_E$.  Moreover,  $G/G_0$ is  isomorphic to $\Gal(\kk_E/\kk_K)$, so $G/G_0\cong\langle\phi_q\rangle$, where $q=q_K$ and $\phi_q$ is the Frobenius of the extension, 
and  it acts on $G_0$ via
$\phi_q\tau\phi_q^{-1}=\tau^q$. Since  $f=f(E/K)=|G/G_0|$,  we obtain that  $\tau^{q^f-1}=1$.

 Let $\upsilon$ be  a representative of $\phi_q$ in $G$; then $\upsilon^f=\tau^r$ for some $r$.
 Taking into account that $\tau^r$ commutes with $\upsilon$, we have that 
 the order of $\tau^r$ divides $q-1$, whence the  order of
$\tau$, $e=e(E/K)$, divides $r(q-1)$. In particular, the Galois group of a
finite tame extension is of the form
\[
   \left\langle \upsilon,\tau \mid \upsilon\tau\upsilon^{-1}=\tau^q,\ \tau^e=1,\ \upsilon^f=\tau^r\right\rangle,
\]
where $e$ divides $\left(q^f-1, r(q-1)\right)$. 

A Galois extension $E/K$ such that $\Gal(E/K)$ is a semidirect product of $G_0$ and a subgroup isomorphic to $G/G_0$ will be called a split extension.
In general, a tame extension $E/K$ need not be split; however, there always 
exists an unramified extension $M$ of $E$ such that $M/K$ 
is split.
In fact, let $M/E$ be  the unramified extension of degree
$s=\sfr{e}{(e,r)}={\rm ord}\,\tau^r$. Then $M=M^{\ur}E$
is Galois over $K$, and $M^{\ur}$ and $E$ are linearly disjoint over
$E^{\ur}$. 
Denote by $\tilde\tau$ the lifting of $\tau$ to
$\Gal(M/K)$ that fixes $M^{\ur}$, and by $\tilde\upsilon$ any lifting of $\upsilon$. We have that $\tilde\tau$ generates the inertia subgroup, and
$\tilde\upsilon$ satisfies $\tilde\upsilon^{fs}=1$. It follows that
$\Gal(M/K)$ is generated by $\tilde\upsilon,\tilde\tau$ with relations
$\tilde\upsilon\tilde\tau\tilde\upsilon^{-1}=\tilde\tau^q$, $\tilde\tau^e=\tilde\upsilon^{fs}=1$, and its
inertia subgroup is generated by $\tilde\tau$. Finally, we remark, for later use,  that  
$[M:E]=s$ is coprime to $p$ and so if the $p$-core of $\Gal(E/K)$ (i.e., the
intersection of all $p$-Sylow subgroups of $\Gal(E/K)$) is trivial, then the same is true for the $p$-core of $\Gal(M/K)$.

\section{The correspondence Theorem}
\label{correspondence}

In this section we will state and prove the key result for the classification of the extensions. It is a direct generalization of Theorem \ref{CP2:theo1}, and it describes, for any $k$, the isomorphism
classes $[E/K]$ of extensions of $E/K$ degree $p^k$ having no intermediate extensions.

Let $F$ be the compositum of all normal and tame extensions of
$K$ whose Galois group is a subgroup of $GL_k(\F_p)$ and let
$H=\Gal(F/K)$.
\begin{remark}\label{Hsplit}
{\rm
Let $P$ be a $p$-Sylow subgroup of $H$ and let $U/K$ be the maximal
unramified subextension of $F/K$. Since $F/K$ is tame, the degree
$[F:U]$ is coprime to $p$. Denote by $U_0$ the unramified extension of
$K$ such that $[U_0:K]=|P|$. Then $U_0$ is the fixed field of a normal
subgroup $A$ of $H$ of order coprime to $p$, and therefore $H$ is isomorphic to a
semidirect product $A\rtimes P$.
}
\end{remark}
The main result of this section is the following. 
\begin{theo}
  \label{CP2:th1}
 There exists a natural one-to-one correspondence between the isomorphism classes of
  extensions of degree $p^k$ ($k\ge 1$) having no intermediate extension and the irreducible
  $H$-submodules $\Xi\subset{F^\times}/{(F^\times)^p}$ of dimension $k$ of the Galois module
  ${F^\times}/{(F^\times)^p}$. The isomorphism class $[L/K]$ corresponds to $\Xi$ where $LF=F(\Xi^{\sfr{1}{p}})$,
  and $\Gal(LF/K)$ is always a split extension of $\Gal(F/K)$.
\end{theo}

One sees that this case is very similar to the case of extensions of degree $p$. It follows
easily that the Galois group of the normal closure of an extension $L/K$ which corresponds to $\Xi$ is isomorphic to the semidirect product of
$\Xi$ by the smallest quotient of $H$ acting on it.

The following lemma will be useful in the proof of the theorem.
\begin{lemma}
  \label{CP2:coho}
  Let $H$ be finite group, and let $A$ be a normal subgroup of $H$  such that $|A|$ is coprime to $p$ and $H/A$ is a $p$-group.  Let $S$ be a finite dimensional irreducible representation of $H$ over $\F_p$ of  dimension $\geq2$. Then any extension of $H$ by $S$ with the given action is isomorphic to the
  split extension $S\rtimes{}H$, and the complements to $S$ in $S\rtimes{}H$ are exactly
  the conjugates of $H$.
\end{lemma}
\begin{proof}
  This is essentially the same argument used in the proof of
  \cite[Thm.~6.1]{bartel2011brauer}. Since $|S|$ and $|A|$ are relatively prime we have
  $H^q(A,S)=0$ for all $q\geq1$ (see \cite[Cor.~1, \S2, Chap.~VIII]{serre1979local}). Consequently,
  we have \cite[Prop.~5, \S6, Chap.~VII]{serre1979local} that 
\[
   0 \rightarrow H^2(H/A, S^A) \rightarrow H^2(H, S) \rightarrow H^2(A, S) = 0,
\]
is an exact sequence, hence $H^2(H/A, S^A)\cong{}H^2(H, S)$. We shall now show that the module $S^A$ of
$A$-invariants is trivial. Suppose that this is not the case: since $H/A$ is a $p$-group, by the class formula, the subspace of elements of $S^A$  fixed by  $H/A$ is non-trivial, so there exists a $H/A$-invariant subspace of dimension 1, contradicting the irreducibility of $S$.
Therefore $S^A=0$ and $0=H^2(H/A, S^A)=H^2(H, S)$, so all the extensions of $H$ by $S$ are
split (see \cite[\S3, Chap.~VII]{serre1979local} for the interpretation of $H^2(H,S)$).

In the same way we obtain that $H^1(H,S)=0$. But $H^1(H,S)$ classifies all possible splittings
$H\rightarrow{}S\rtimes{}H$ of the canonical projection up to conjugacy (see \cite[Prop.~2.3, \S2,
Chap.~IV]{brown1982cohomology}), so the only possible complements of $S$ are the conjugates of $H$.
\end{proof}

\begin{proof}[Proof of Theorem \ref{CP2:th1}]
  We first show that the compositum $L_F=LF$ is an abelian elementary
  $p$-extension of degree $p^k$ of $F$. The degree $[L_F:F]$ is $p^k$,
  since $L$ and $F$ are linearly disjoint over $K$.  The extension
  $L/K$ is totally ramified, since otherwise it would contain an
  unramified subextension of degree $p$ and hence it would have a
  proper subextension.  Let $G=\Gal(\tilde{L}/K)$ be the Galois group
  of the normal closure $\tilde L$ of $L/K$, and let
  \[
    G = G_{-1} \supseteq G_0 \supseteq G_1 \supseteq \dots
  \]
  be its lower numbering ramification filtration. It is well-known that for every $i$ the subgroup $G_i$
  is normal in $G$, and that for $i\geq1$ the quotient
  ${G_i}/{G_{i+1}}$ is an elementary abelian $p$-group.

  Let $\tilde{H}\subseteq{}G$ be the subgroup fixing $L$.  Since $L/K$
  has no intermediate extension, $\tilde{H}$ is a maximal subgroup of
  $G$ and hence there is a unique index $t$ such that
  $G_{t+1}\subseteq{}\tilde{H}$ and  $G_t\tilde{H}=G$; observe that
  $t\geq 1$ since $L/K$ is a totally ramified wild extension.

  Moreover, since
  $\tilde{L}$ is the normal closure of $L/K$, no subgroup of
  $\tilde{H}$ is normal in $G$.  It follows that $G_{t+1}$ and $\Core_G(\tilde{H})$ (that is, the
  intersection of all conjugates
  $\bigcap_{\sigma\in{}G}\sigma{}\tilde{H}\sigma^{-1}$ of $\tilde H$) are trivial.
  The centralizer $C_{\tilde{H}}(G_t)$ 
  is trivial too, since it is contained in all the conjugates of
  $\tilde{H}$. It follows that $G_t$ is a faithful $\tilde{H}$-module. 

  Observe now that, since $G_{t+1}$ is trivial, $G_t$ is an elementary
  abelian $p$-group; from $C_{\tilde{H}}(G_t)=\{1\}$ we then obtain
  $G_t\cap{\tilde{H}}=\{1\}$ and therefore $G$ is a semidirect product
  $G_t\rtimes \tilde{H}$ and $|G_t|=p^k$. This implies that the module
  $G_t$ is irreducible: in fact, if there would exist a proper submodule
  $A$ of $G_t$, then $A\rtimes \tilde{H}$ would be a proper subgroup of $G$
  containing $\tilde{H}$, contradicting the maximality of $\tilde{H}$.

We will next show that the $p$-core $O_p(\tilde{H})$ is
trivial. Suppose $O_p(\tilde{H})\neq{}1$; since
$O_p(\tilde{H})\vartriangleleft{}\tilde{H}$, the subgroup
$G_t\rtimes{}O_p(\tilde{H})$ is normal in $G$. Now, $O_p(\tilde{H})$
is a $p$-group and therefore the set $A$ of the points of $G_t$ which
are fixed by the action of $O_p(\tilde{H})$ is a non-trivial subgroup
of $G_t$, different from $G_t$ because the action is
faithful. However, it is immediate to see that the subgroup 
$A$ is also invariant by $\tilde{H}$, contradicting the irreducibility of $G_t$.

Let $L_1$ be the subfield of $\tilde{L}$ fixed by $G_t$; we have that
$\Gal(L_1/K)\cong{}\tilde{H}$ and $L_1/K$ is a tame extension, since
the $p$-core of $\tilde{H}$ trivial. Clearly its Galois group embeds into
$\Aut(G_t)\cong GL(k,\F_p)$, so we have $L_1\subseteq{}F$ and
$\tilde{L}\subseteq{}L_F$.

By construction, $\tilde{L}/L_1$ is totally ramified and hence $\tilde
L \cap F=L_1$. It follows that $\Gal(L_F/F) \cong G_t$ and therefore
$L_F/F$ is abelian elementary of degree $p^k$.  By Kummer
theory, $L_F$ can be written as $F(\Xi^{\sfr{1}{p}})$ for some
$\Xi\subset{F^\times}/{(F^\times)^p}$. This is naturally a $H$-module; this module
is irreducible: in fact, a proper submodule of $\Gal(L_F/F)$ would fix
a proper subextension $E$ of $L_F/F$, normal over $K$; hence, $E\cap\tilde{L}$ would be a
proper subextension of $\tilde{L}/L_1$ normal over $K$, corresponding to
a proper $\tilde{H}$-submodule of 
the irreducible module $\Gal(\tilde{L}/L_1)\cong G_t$. 
\smallskip

Conversely, we show that each irreducible $\Xi$ of dimension $k$ corresponds
to one isomorphism class $[L/K]$. Let $M=F(\Xi^{\sfr{1}{p}})$, then
$S=\Gal(M/F)$ is an irreducible $H$-module by Kummer theory. Let ${\mathcal G}=\Gal(M/K)$; then  ${\mathcal G}$ is  an extension of $H$ by $S$. 

By Remark \ref{Hsplit}, if $k\ge2$ then the group $H$ satisfies the hypothesis of Lemma \ref{CP2:coho}; if $k=1$ we have $(|H|,|S|)=1$. In any case it follows that the exact sequence 
$$
1 \to S \to {\mathcal G} \to H \to 1
$$
splits, namely, there exists a complement $B$ to $S$ in ${\mathcal G}$
and ${\mathcal G} \cong S\rtimes B$; moreover, the complements
to $S$ in ${\mathcal G}$ are exactly the conjugates of $B$.

If $L$ is the fixed field of $B$, then the fixed fields of the conjugates
of $B$ form an isomorphism class of extensions of $K$ of degree
$p^k$.  The field $L$, as well as its conjugates, has no proper
subextension.  In fact, let $K\subseteq L_0\subseteq L$ and let
$C={\rm Gal}(M/L_0)$; then $B\subseteq C$ and $C\cong S_0\rtimes B$
where $S_0=S\cap C$, by the formula $|S_0|\cdot|B|=|C|$. But this
means that $S_0$ is $H$-invariant, and therefore either $S_0=\{1\}$ or
$S_0=S$.
\smallskip

It is now immediate to check that the two maps are inverse to each other. 
\end{proof}

\bigskip

Looking carefully at the proof of the theorem, one can see that it works also for many other choices of the field $F$. For instance, we have the following corollary.
\begin{corollary}\label{corteo}
Suppose that Theorem \ref{CP2:th1} holds replacing $F$ with a field ${\mathcal F}$. Then it holds also for the field $\tilde{\mathcal F}$ in the following cases:\\
 \ \ i) $[\tilde{\mathcal F}:{\mathcal F}]$ is finite and coprime to $p$;\\
 ii) $\tilde{\mathcal F}$ is the subfield of ${\mathcal F}$ fixed by the $p$-core of $\Gal({\mathcal F}/K)$.
\end{corollary}

\begin{remark}\label{split}
{\rm  The last corollary and the argument given at the end of Section \ref{sec:2} 
show that there exists a field $F$  for which  Theorem \ref{CP2:th1} holds with the additional properties that  $F$
is a split extension of $K$ and the $p$-core of $H=\Gal(F/K)$ is trivial.

Unless otherwise specified, in the following  $F$ will denote an extension of $K$ with the properties just mentioned.
}

\end{remark}

\section{The Galois structure of $F^\times/F^{\times p}$}
\label{section4}

In this section we will describe the $H$-module structure of $F^\times/F^{\times p}$. We start with the classification of the irreducible representations of $H$ over $\F_p$, via the study of the irreducible representations over $\bar \F_p$ (Section \ref{CP2:tamerep}); our results are quite similar to those  of \cite{chebolu2009reciprocity}. Next, we determine the decomposition of $F^\times/F^{\times p}$ as a $\F_p[H]$-module; we first obtain the decomposition of the $\bar\F_p[H]$-module obtained  from $F^\times/F^{\times p}$ by extension of scalars and then we bring back this information to get the structure of $F^\times/F^{\times p}$.

\subsection{The irreducible representations of a tame Galois group}
\label{CP2:tamerep}

We shall consider groups of type ${\mathfrak H}=T\rtimes U$, where  $T=\langle\tau\rangle$ and $U=\langle\upsilon\rangle$ are cyclic 
groups of orders $e$ and $f$, respectively,  with $(e,p)=1$ and $e\mid{}q^f-1$, and 
where $\upsilon$ acts on $T$ via the map $x\mapsto{}x^q$, 
 and $q$ is a power of $p$, that in view of our applications we will denote by $p^{f_K}$. Note that he group $H=\Gal(F/K)$ belongs to this type of groups with $T=H_0$.

 Let $V$ be an $\bar\F_p$-vector space and let $\rho\, : \, {\mathfrak H} \to GL(V)$ be a representation of ${\mathfrak H}$ irreducible over $\bar\F_p$. 
 Let $T_0={\rm ker}(\rho_{|T})$, $\bar T=\bar T(\rho)=T/T_0$; to study the irreducible representation $\rho$ we can study the irreducible representation $\bar\rho\,:\, \bar T\rtimes U\to GL(V)$ which is faithful  on $\bar T.$

Denote by $\tilde{U}$ be the kernel of $U\rightarrow\Aut(\bar T)$, {\it i.e.}, $\tilde U=C_U(\bar T)$.  
Let $t=t(\rho)=|\bar T|$, $r=r(\rho)={\rm ord}_t^\times p$ (that is, the smallest power of $p$ such that $t|p^r-1$).
Let $s=s(\rho)= {\rm ord}_t^\times q$; hence $s=r/(r,f_K)$ and $q^s$ is the smallest power of $q$ such that $t|q^s-1$. Then $\tilde U$ is generated by $\upsilon^s$, that we denote by  $\tilde\upsilon$. 

The group $\bar T\times \tilde U$ is abelian, so its $p$-core is equal
to its $p$-Sylow subgroup. With the same argument as in the proof of
Theorem \ref{CP2:th1}, we get that this subgroup acts trivially on
$V$, because otherwise the representation would have a proper
invariant subspace. It follows that the representation of $\bar
T\times \tilde U$ can be factored by the $p$-Sylow, hence it defines a
representation of an abelian group of order coprime to $p$; so there
exists a subspace of dimension $1$ invariant for $\bar T\times \tilde
U$ and therefore also for $T\rtimes \tilde U.$

Let $V_\chi$ be a 1-dimensional invariant subspace,  where 
$\chi:T\rtimes\tilde{U}\rightarrow\bar\F_p^\times$ is the character of the representation in $V_\chi$. Let
$\alpha=\chi(\tau)$ and $\beta=\chi(\tilde\upsilon)$;  then
$\upsilon^i$ maps $V_\chi$ to a space where $\tau$ acts via a
conjugate $\alpha^{q^i}$ of $\alpha$, and the orbit under $\upsilon$
generates the whole $V$. Moreover, $\alpha$  has order $t$,
because the representation is faithful on $\bar T$. Hence
for $0\leq{}i<s$ the $\upsilon^iV_\chi$ are all distinct, and $\rho$ is equal to $\Ind_{T\rtimes\tilde{U}}^{T\rtimes{}U}(V_\chi)$.

If $x\in{}V_\chi$ is a generator, then the elements
$x,\upsilon{}x,\dots,\upsilon^{s-1}x$ are a basis of $V$ on
which $\tau$ and $\upsilon$ act via the matrices:
\[
\mathcal{T}_\alpha =
\begin{pmatrix}
 \alpha & & & & \\
 & \alpha^q\! & & & \\
 & & \alpha^{q^2}\!\!\! & & \\
 & & & \ddots & \\
 & & & & \!\alpha^{q^{s-1}}\! \\     
\end{pmatrix},\qquad
\mathcal{U}_\beta =
\begin{pmatrix}
  & & & & \beta \\
  1 & & & & \\
  & 1 & & & \\
  & & \ddots & & \\
  & & & 1 & \\
\end{pmatrix}.
\]
The representations obtained in this way starting from $V_\chi$ or
$\upsilon^i{}V_\chi$ are equal for all $i$, and on $\upsilon^i{}V_\chi$
the action is described via the character $\chi^{q^i}$ (the character obtained from $\chi$ by conjugation with $\upsilon^i$).

On the other hand,  given any subgroup $T_0$ of $T$ and any character $\chi$ of $T\rtimes \tilde U$, where $\tilde U=C_U(\bar T)$ ($\bar T=T/T_0$) and  ${\rm ker}(\chi_{|T})=T_0$,  all representations of the type $\Ind_{T\rtimes\tilde{U}}^{T\rtimes{}U}(V_\chi)$
 are irreducible, 
because clearly the matrices $\mathcal{T}_\alpha$ and  $\mathcal{U}_\beta$ do not have common proper invariant subspaces.

The results just proven can be summarized in the following proposition.

\begin{prop}
\label{irrfpbar}

 Let $T_0$ be any subgroup of  $T$ and let $\bar T=T/T_0$. All irreducible representations $\rho$ of $T\rtimes{}U$ in a
  $\bar\F_p$-vector space $V$  such that ${\rm ker}(\rho_{|T})=T_0$  are obtained as
  $\Ind_{T\rtimes\tilde{U}}^{T\rtimes{}U}(V_\chi)$, where  $\tilde U=C_U(\bar T)$, $\chi$  is a 
  character of $T\rtimes\tilde{U}$ such that ${\rm ker}(\chi_{|T})=T_0$ and $V_\chi$ is the representation given by $\chi$.  
  
 Moreover, $\Ind_{T\rtimes\tilde{U}}^{T\rtimes{}U}(V_\chi)= \Ind_{T\rtimes\tilde{U}}^{T\rtimes{}U}(V_{\chi'})$  if and only if 
$\chi$ and $\chi'$ are conjugate.  

The $p$-core of $T\rtimes U$ is contained in the kernel of all irreducible representations.
\end{prop}

Our next step is to show how to deduce the irreducible representations over $\F_p$ from those over $\bar\F_p.$
More generally, if $\Omega/\Omega_0$ is a Galois extension of fields, we deduce the irreducible representations of a group $G$ over $\Omega_0$, from 
the irreducible representations of $G$ over $\Omega$.

Let $\rho:G\rightarrow{}GL(V)$ be an irreducible representation of  $G$ in $V=\Omega^n$ over  $\Omega$.  The irreducible representation over $\Omega_0$ which contains $\rho$ is 
the sum of the conjugates of $\rho$ under the action $\Gal(\Omega/\Omega_0)$. In fact, a
 representation containing $\rho$, defined over $\Omega_0$, must contain all the conjugates of $\rho$ under the action of $\Gal(\Omega/\Omega_0)$ and it is easy to verify that the sum over all distinct conjugates is defined and it is irreducible over $\Omega_0$. 
It follows that the dimension of this representation is the product 
of the dimension of $\rho$ by the number of its conjugates. 

We apply this argument to the representations described in Proposition \ref{irrfpbar}:  consider $\Ind_{T\rtimes\tilde{U}}^{T\rtimes{}U}(V_\chi)$, and set
$\alpha=\chi(\tau)$ and $\beta=\chi(\tilde\upsilon)$ as above. Clearly this representation
is stabilized by $\Gal(\bar\F_p/\F_p(\alpha,\beta))$ and by the
powers of the Frobenius $\phi_q$ fixing $\beta$, since $\chi$ and $\chi^{q^i}$ correspond to the same representation.
In $\Gal(\F_p(\alpha)/\F_p)$ the group generated by $\phi_q$ is equal to
the group generated by $\phi_{p^{(r,f_K)}}$. Putting $w=[\F_p(\beta):\F_p]$, we obtain that the
smallest power of $\phi_{p^{(r,f_K)}}$ fixing $\F_p(\beta)$ is $\phi_{p^{\lcm(w,(r,f_K))}}$, which generates also the
stabilizer of the representation. Hence the number of conjugates of $\rho$ is 
$\lcm(w,(r,f_K))$. We can now prove the following proposition. 

\begin{prop}
\label{CP2:repdim}
Let $\rho\,:\,T\rtimes{}U\to GL(V)$ be an irreducible representation  over $\F_p$. Consider the restriction of the representation to $T\rtimes{}\tilde{U}$
over the algebraic closure, where, setting ${\rm ker}(\rho_{|T})=T_0$ and $\bar T=T/T_0$, $\tilde U=C_U(\bar T)$. Let 
$\chi$ be the character of an invariant
subspace of dimension $1$.  Then
\[
   \dim_{\F_p}V = \lcm(\sfr{rw}{(r,f_K)},r)
\]
where $r=\ord_t^\times(p)$, and $w=[\F_p(\beta):\F_p]$ where $\beta=\chi(\tilde\upsilon)$.
\end{prop}
\begin{proof}
Indeed, $\Ind_{T\rtimes\tilde{U}}^{T\rtimes{}U}(V_\chi)$ has dimension $s$ so we obtain
\[ 
   \dim_{\F_p}V = s\cdot\lcm(w,(r,f_K)),
\]
and the proposition follows because $s=\sfr{r}{(r,f_K)}$.
\end{proof}

\subsection{The structure of \texorpdfstring{${F^\times}/{(F^\times)^p}$}{[F\^{}*]\_F}}

In this section $F/K$ will be a normal tamely ramified extension, and  $H$ will denote the group $\Gal(F/K)$.
We  also assume that  $\zeta_p\in F$ and that the extension is split. Let $H_0=T=\langle\tau\rangle$, 
and let $\upsilon$ be a lifting of the
Frobenius $\phi_q$ generating a complement $U$, so that $H=T\rtimes{}U$.  Let $F'$ be the field fixed by
$U$; then the extension $F/F'$ is unramified. 
Let $\pi$ be a uniformizer of $F'$, that we choose as an $e$-th root of a uniformizer of $K$, where $e=e(F/K)$
(recall that in this case $e|q-1$). We observe that $\pi$ is also a uniformizer of $F$. 
Denote by $U_F$ the group of units of $F$ and by $U_{1,F}$ the subgroup of principal units of $F$.
It is well-known that 
\begin{equation}
\label{f*}
F^\times\cong\left(\langle\pi\rangle\times\kk_F^\times\right)\oplus U_{1,F}
\end{equation}
as $H$-modules, where as usual $U_{1,F}$ denotes the group of principal units of $F$. 

We now consider the filtration $\{ U_{i,F}\}_{i\ge1}$ of $U_{1,F}$,
where $U_{i,F}=\{u\in U_F\, |\, u\equiv 1\pmod{\pi^i}\}$. We want to  describe the structure of $U_{i,F}$ as $H$-module.

For
$i\geq1$, the action of $\tau$ and $\upsilon$ on $U_{i,F}/U_{i+1,F}$ is described by
\[
    \tau(1+\alpha\pi^i) = 1+\zeta^i\alpha\pi^i+\dots,\qquad 
    \upsilon(1+\alpha\pi^i) = 1+\alpha^q\pi^i+\dots,
\]
for each $\alpha\in{}U_F$, where $\zeta=\tau(\pi)/\pi$ is a primitive $e$-th root of $1$.

We can identify $U_{i,F}/U_{i+1,F}$ with $\kk_F$ via the map
\[
    (1+\alpha\pi^i) \mapsto \bar\alpha
\]
and this induces on $\kk_F$ the action given by
\[
    \tau(\bar\alpha) = \bar\zeta^i\bar\alpha,\qquad \upsilon(\bar\alpha) = \bar\alpha^q.
\]

\begin{prop}
\label{CP2:Uiprojective}
For $i\geq1$, let $M_i$ be the $\kk_K[H]$-module formed by $\kk_F$ as set, and with the above
action. Then $M_i$ is projective both as a $\kk_K[H]$-module and as a $\F_p[H]$-module.
\end{prop}
\begin{proof}
  Consider the sum $M=\oplus_{i=1}^{e}M_i$. We shall prove that $M$ is a free $\kk_K[H]$-module.
  Let $\bar\eta$ be normal basis generator 
  for $\kk_F$ over $\kk_K$; we claim that the vector $w=(\bar\eta)_{1\leq{}i\le e}$, with all components equal to
  $\bar\eta$, is a generator for $M$. Indeed, under the action of $\kk_K[U]$  we obtain all elements of the form
  $(\bar\alpha)_{1\leq{}i\le e}$; now the vector
  \[
     \frac{1}{e}\sum_{j=1}^{e}\tau^j\left((\bar\zeta^{-jk}\bar\alpha)_{1\leq{}i\le e}\right)
  \]
  has the  $k$-th component equal to  $\bar\alpha$ (recall that $e|q-1$), and the other
  components equal to $0$. This proves the claim, namely $M=\kk_K[H]w$. Hence $M$ is a quotient of the free module $\kk_K[H]$, and in fact $M=\kk_K[H]$ since they have the same dimension. 
 It follows that, for $1\le i\le e$, the module  $M_i$ is projective, being a direct summand of a
  free module. On the other hand, $M_{e+i}=M_i$ and thus $M_i$ is projective for all $i\ge1$. Clearly, this implies that $M_i$ is projective also as a $\F_p[H]$-module.
  \end{proof}

The proposition just proved shows that $U_{i+1,F}$ has a complement in $U_{i,F}$ as a $\F_p[H]$-module, namely, 
$U_{i,F}\cong M_i\oplus U_{i+1,F}.$ 
By induction 
\begin{equation}
\label{umi}
U_{1,F}\cong \bigoplus_{i=1}^\frac{pe_F}{p-1}M_i\oplus U_{\frac{pe_F}{p-1}+1,F}.
\end{equation}

\begin{prop}
\label{IF}
Let $I_F=\sfr{pe_F}{(p-1)}$, and let $\I{0,I_F}$ be the set of integers prime to $p$ in the interval
$[0,I_F]$.
Denote by $\F_p$ and  $M_\omega$ the $\F_p[H]$-module  $\F_p$ with the trivial action and the action given by  the cyclotomic character $\omega$, respectively. Then we have the $\F_p[H]$-module isomorphism
$$
   {F^\times}/{(F^\times)^p} \cong \F_p \oplus \big(\bigoplus_{i\in\I{0,I_F}} M_i\big) \oplus M_\omega.
$$
\end{prop}
\begin{proof}
By equation \eqref{f*} we get 
$$
{F^\times}/{(F^\times)^p} \cong \left(\langle\pi\rangle\times\kk_F^\times\right)/\left(\langle\pi\rangle\times\kk_F^\times\right)^p\oplus U_{1_F}/U_{1,F}^p\cong \F_p\oplus U_{1_F}/U_{1,F}^p,$$
since clearly the action on $\left(\langle\pi\rangle\times\kk_F^\times\right)/\left(\langle\pi\rangle\times\kk_F^\times\right)^p$ is trivial.

Consider now $U_{1_F}/U_{1,F}^p$. By equation \eqref{umi}, 
$$U_{1,F}^p\cong \bigoplus_{i=1}^\frac{pe_F}{p-1}M_i^p\oplus U_{\frac{pe_F}{p-1}+1,F}^p$$
and, applying \cite[(5.7) and (5.8)]{Fesenko_Vostokov}, we get 
$$U_{1,F}^p\cong \bigoplus_{i=1}^{\frac{e_F}{p-1}-1}M_{ip}\oplus N \oplus U_{\frac{pe_F}{p-1}+1,F},$$
where $N$ s a subspace of $M_\frac{pe_F}{p-1}$ of codimension 1.

It follows that 
$$U_{1_F}/U_{1,F}^p=\bigoplus_{i\in\I{0,I_F}} M_i \oplus M_\frac{pe_F}{p-1}/N.$$ 
Finally, the submodule 
$M_\frac{pe_F}{p-1}/N$ corresponds via Kummer theory to the unramified extension of
degree $p$, and the action of $H$ on this submodule is given by the cyclotomic character (see for instance \cite[Prop.7]{delcorso2007compositum}). 
\end{proof}

By the last proposition we are left to decompose each $M_i$ (for $i\in\I{0,I_F}$), and hence $Y:=\oplus_{i\in\I{0,I_F}} M_i$,  into a sum of irreducible representations. 
We begin by studying the decomposition of  $M_i \otimes_{\F_p} \bar\F_p$; using Proposition \ref{CP2:repdim} we shall then obtain the decomposition of $M_i$.

\begin{prop}
\label{CP2:Uistruct}
Let $V_i$ be the $\bar\F_p[T]$ module of dimension $1$ where $\tau$ acts via multiplication by
$\bar\zeta^i$. Then
\[
  M_i \otimes_{\kk_K} \bar\F_p  \cong \Ind_T^H(V_i)\quad {\rm and}\quad M_i \otimes_{\F_p} \bar\F_p  \cong \left(\Ind_T^H(V_i)\right)^{f_K}\]
as $\bar\F_p[H]$-modules.
\end{prop}
\begin{proof}
Consider the $\kk_F[H]$-module $M_i\otimes_{\F_p}\kk_F$ with the action $h(\alpha\otimes\beta)=h\alpha\otimes\beta$.
Denote by $M_i'$ the $\kk_F[T]$-module obtained from $M_i$ by restrictions of the scalars and consider the  $\kk_F[H]$-module  $ \kk_F[H] \otimes_{\kk_F[T]} M_i$.

It is easy to check that the map   
\begin{align*}
	M_i \otimes_{\kk_K}\kk_F &\stackrel{\sim}{\longrightarrow} 
    \kk_F[H] \otimes_{\kk_F[T]} M_i' \\
  \alpha\otimes \beta &\mapsto \sum_{i=0}^{f-1} \upsilon^{-i} \otimes \upsilon^i(\alpha)\beta
\end{align*}
is an isomorphism of $\kk_F[H]$-modules. We now tensor both sides with $\bar\F_p$; taking into account that $M_i'\otimes_{\kk_F}\bar\F_p=V_i$ and that 
$\kk_F[H] \otimes_{\kk_F[T]} V_i \cong \Ind_T^H(V_i)$ 
we get the first formula; the second one  follows immediately.
\end{proof}

Let $T_0$ be the kernel of the action of $\rho_{|T}$, let $\bar T=T/T_0$ and
let $\tilde{U}$ be the centralizer of $\bar T$ in $U$. Assume that $\tilde{U}$ is generated by $\tilde{\upsilon}$ of order $\tilde{u}$.
Then
\[
   \Ind_T^{T\rtimes{}\tilde{U}}(V_i) \cong \bar\F_p[\tilde{U}] \otimes _{\bar \F_p}V_i=\bigoplus_{\{\beta|\beta^{\tilde{u}}=1\}} V_{(i,\beta)},
\]
where $V_{(i,\beta)}$ is the representation of dimension $1$ such that $\tau$,
$\tilde{\upsilon}$ act by multiplication by $\zeta^i$ and $\beta$. Consequently we have
\[
    \Ind_T^{H}(V_i) = \bigoplus_{\{\beta\mid\beta^{\tilde{u}}=1\}} 
          \Ind_{T\rtimes{}\tilde{U}}^H (V_{(i,\beta)})
\]
and, by Proposition \ref{irrfpbar}, this  is a decomposition into irreducible representations.
Denoting by $\bar M_i:= M_i \otimes_{\F_p} \bar\F_p$, by Proposition \ref{CP2:Uistruct} we have 
$\bar M_i  \cong \left(\Ind_T^H(V_i)\right)^{f_K}$ and hence
$$\bar Y:=\bigoplus_{i\in{}\I{0,I_F}}\bar{M}_i=
   \bigoplus_{i\in{}\I{0,I_F}}  \bigoplus_{\{\beta\mid\beta^{\tilde{u}}=1\}} 
         ( \Ind_{T\rtimes{}\tilde{U}}^H (V_{(i,\beta)}))^{f_K}.$$

We now compute the multiplicity of each irreducible representation contained in $\bar Y.$ 

Observe that if $\zeta^i=\zeta^j$ then $V_{(i,\beta)}=V_{(j,\beta)}$; let $V_{(\alpha, \beta)}$ denote any $V_{(i,\beta)}$ such  that $\zeta^i=\alpha$ and set $J_{(\alpha, \beta)}= \Ind_{T\rtimes{}\tilde{U}}^H (V_{(\alpha,\beta)})$.

For any   $e$-th root of unity $\alpha\in\bar \F_p^\times$, the indices $i$ such that $\zeta^i=\alpha$ are a congruence class modulo $e$ and, among these, there are exactly $p-1$ classes modulo $ep$ which are coprime to $p$. Hence, in the set $\I{0,I_F}$ there are $\frac{p-1}{ep}\frac{e_Fp}{p-1}=e_K$ such indices. 
It follows that the representation $J_{(\alpha,\beta)}$ appears
$e_Kf_K=[K:\Q_p]$ times. Finally, since  the conjugate pairs $(\alpha^{q^i},\beta)$ yield
the same representation,
the multiplicity of $J_{(\alpha, \beta)}$ in $\bar Y$ is $se_Kf_K$, where $s=[U:\tilde U].$

Summarizing, we obtain the following proposition. 

\begin{prop}
\label{CP2:multiplicity}
Let $\alpha, \beta\in \bar\F_p$. Suppose that $\alpha^e=1$ where
$e=e(F/K)$, and that ${\rm ord}(\alpha)=t$. Let $s=\ord_q^\times(t)$,
$f=f(F/K)$, and $\tilde{u}=f/s$ and suppose that
$\beta^{\tilde{u}}=1$. Then the multiplicity of $J_{(\alpha,\beta)}$
in $\bar Y$ is equal to $s\cdot[K:\Q_p]$.
\end{prop}

\begin{corollary}
\label{regular}
The $H$-module $\bar{Y}$ is isomorphic to $\bar\F_p[H]^{[K:\Q_p]}$. 
\end{corollary}

\begin{proof} The $H$-module $\bar\F_p[H]$ corresponds to the regular representation of $H$, 
and hence each irreducible representation occurs in $\bar\F_p[H]$ with a multiplicity 
equal to its degree. The corollary follows by noticing that the multiplicities of all
irreducible representations in $\bar{Y}$ are equal to their degrees multiplied by 
$[K:\Q_p]$.
\end{proof}

\smallskip

\begin{remark}\label{oss}
{\rm
Let $D=D_{(\alpha, \beta)}$ be the field of definition of $J_{(\alpha,\beta)}$; by Proposition \ref{CP2:repdim} its degree over
$\F_p$ is equal to $d=\lcm(w,(r,f_K))$.
Let $V$ be the irreducible sub-representation of $Y$ defined over $\F_p$,  obtained by summing all the conjugates of $J_{(\alpha,\beta)}$ over $\F_p$; then its multiplicity in $Y$ is $s\cdot[K:\Q_p]$, while its degree is  $s\cdot\lcm(w,(r,f_K))$.
In particular,
the $H$-module $Y$ is not isomorphic, in general, to $\F_p[H]^{[K:\Q_p]}$ whereas this was true in the case of extensions of degree $p$, considered in \cite{delcorso2007compositum} and \cite{dalawat2010serre}.
 }
\end{remark}

For each irreducible $\F_p[H]$-module $V$ contained in $Y$, we now count
the number $n_V$ of submodules isomorphic to $V$ contained in $Y$.
As before, $V$ is the sum of the conjugates of  $J_{(\alpha,\beta)}$ for some $\alpha,\beta$; let $D$ be  the field 
of definition of $J_{(\alpha,\beta)}$, let $[D:\F_p]=d$, and call $m$ the common multiplicity of  $J_{(\alpha,\beta)}$ in $\bar Y$ and of $V$ in $Y$. Then,  all $\F_p$-irreducible 
modules isomorphic to $V$ are contained in $V^m$  
and $n_V$ is equal to the number of $\F_p[H]$-embeddings of $V$ into $V^m$ divided by the number of embeddings with the same image. Now, the $\F_p[H]$-embeddings of $V$ into $V^m$ are in one to one correspondence with the $D[H]$-embeddings of $J_{(\alpha,\beta)}$ into $(J_{(\alpha,\beta)})^m$, associating to $\varphi:\, V\rightarrow V^m$  the restriction to $J_{(\alpha,\beta)}$ of the extension of $\varphi$ to $V\otimes_{\F_p} D.$ Moreover, two embeddings of $V$ have the same image if and only if the same is true for the corresponding embeddings of $J_{(\alpha,\beta)}$. 
Therefore, we count the $D[H]$-embeddings 
\[
 \psi:\, J_{(\alpha,\beta)} \rightarrow (J_{(\alpha,\beta)})^m.
\]
By Schur Lemma the projection of each component of $\psi$ is the multiplication by an element $d_i\in D$, and $\psi$  is injective if and only if the $d_i$ are not all zero. Moreover, 
two embeddings have the same image if and only if they differ by multiplication by a
constant. It follows that the total number of $D[H]$-embeddings is  $|D|^{m}-1$, while the possible constants
are $|D|-1$. As a conclusion, we have proved the following proposition:

\begin{prop}
\label{CP2:submods}
Assume that the irreducible representation $J_{(\alpha,\beta)}$ has
multiplicity $m$ in $\bar Y$. Assume that the field of 
definition of $J_{(\alpha,\beta)}$ is $D$, and let $d=[D:\F_p]$. Then
the number of representations defined over $\F_p$ and containing a
representation isomorphic to $J_{(\alpha,\beta)}$ is
$(p^{dm}-1)/(p^d-1)$.
\end{prop}

\section{Extensions with a prescribed Galois group}
\label{section5}
In this section we classify the isomorphism classes of the extensions of degree $p^2$ of $K$ without
intermediate extensions. 
Actually, we give a more precise result, namely,  for each possible Galois group $G$, we count the number of isomorphism classes of the extensions whose normal closure has a Galois group isomorphic to $G$.

We start by recalling some relevant facts proved above.

By Theorem \ref{CP2:th1},  the isomorphism classes correspond to the irreducible $H$-submodules of dimension 2 of ${F^\times}/{(F^\times)^p}$, where $F$  is as in Remark \ref{split} and $H=\Gal(F/K)$.
An irreducible submodule $X\subset{}{F^\times}/{(F^\times)^p}$ corresponds to
an isomorphism class $[L]$ of extensions of $K$ and, denoting by $\tilde L$  the normal closure of $L/K$, we have  
$$\Gal(\tilde L/K)\cong X\rtimes{}(H/\ker(\rho)),$$
 where $\rho:H\rightarrow\Aut(X)$ is the restriction of the Galois action to $X$. 
We shall see that the isomorphism class of the Galois group  depends only on the group $H/\ker(\rho)$; more precisely, we have the following lemma. 
\begin{lemma}
\label{semidir}
Let $X$ and $X'$ be irreducible submodules of ${F^\times}/{(F^\times)^p}$ of dimension 2. Let $\rho$ and
$\rho'$ be the restrictions of the Galois action of $H$ to $X$ and $X'$, respectively. Then
$$X\rtimes{}(H/\ker(\rho))\cong X'\rtimes{}(H/\ker(\rho')) \iff H/\ker(\rho)\cong H/\ker(\rho').$$
\end{lemma} 
\begin{proof}
The {\it if}\,  part is obvious.  
To prove the converse it is enough to show that if $H/\ker(\rho)\cong H/\ker(\rho')$ then, by identifying the two groups with their images in $GL_2(\F_p)$, they are conjugate. 
Let $\rho_0$ (resp. $\rho'_0$) be an irreducible component of $\rho$ (resp. $\rho'$) over $\F_{p^2}$. 
With a case by case analysis we will show below that $H/\ker(\rho_0)$ and $H/\ker(\rho'_0)$ are conjugate over $\F_{p^2}.$ Clearly, this implies that 
$H/\ker(\rho)$ and $H/\ker(\rho')$ are conjugate over $\F_{p^2}$; now, since $H/\ker(\rho)$ and $H/\ker(\rho')$  are realizable over $\F_p$,  they are conjugate over $\F_p$.
\end{proof} 

As before, we shall write   
 $H=T\rtimes{}U$, where $T=\langle\tau\rangle$ is
the inertia subgroup and $U=\langle\upsilon\rangle$. 

Each irreducible representation $\rho$ of $H$ is described by two matrices $\rho(\tau)={\mathcal T}_\alpha$, $\rho(\upsilon)={\mathcal U}_\beta$  for  some $\alpha,\beta\in\bar\F_p^\times$. 
Moreover, $\alpha$ and $\beta$ must satisfy the following conditions: let $t$ be the order of $\alpha$, $r=\ord_t^\times(p)=[\F_p(\alpha):\F_p]$, $s=\ord_t^\times(q)$ and $w=[\F_p(\beta):\F_p]$; then 
\begin{equation}
\label{conditions}
t\mid e\quad {\rm and} \quad \ord(\beta) \mid \frac fs \,.
\end{equation}

According to Proposition \ref{CP2:repdim}, the degree 
of the representation given by the matrices ${\mathcal T}_\alpha$ and ${\mathcal U}_\beta$ 
is $\lcm(\sfr{rw}{(r,f_K)},r)$. If follows that, if the degree of $\rho$ is 2, then necessarily $\alpha,\beta\in\F_{p^2}^\times$; moreover, $s=\sfr{r}{(r,f_K)}=1,2$, and the case $s=2$ can occur only if $w=1$, i.e.,
$\beta\in\F_p^\times$. 
Now, our construction of the field $F$ is such that both $e$ and $f$ are divisible by $p^2-1$
(in fact, it is immediate to check that $F$ contains the unramified extension of $K$ of degree
$p^2-1$ and the splitting field of $X^{p^2-1}-\pi_k$); hence every pair $(\alpha,\beta)$ satisfying
the condition $\lcm(\sfr{rw}{(r,f_K)},r)=2$ satisfies \eqref{conditions} as well.

The representations of degree 2 do not occur in the submodules $\F_p$ and $M_\omega$ of 
Proposition \ref {IF}; hence, by Proposition \ref{CP2:multiplicity}, their multiplicity in $F^\times/(F^\times)^p$ is $s\cdot[K:\Q_p]$.

For convenience of the reader, we state without proof the following simple lemma. 
\begin{lemma}
\label{psiab}
For any ordered pair $(a,b)$ of natural numbers, let $\psi(a,b)$ be the number of elements of 
order $a$ in the group $\Z/a\Z\times{}\Z/b\Z$. Then 
\begin{equation}
\label{CP2:eqsubgr}
  \psi(a,b) = a\cdot(a,b) \cdot \prod_{\substack{\ell\text{ prime}\\\ell\mid a/(a,b)}}
      \left(1-\frac{1}{\ell}\right)
   \cdot \prod_{\substack{\ell\text{ prime}\\\ell \mid a,\ \ell\nmid
      a/(a,b)}}\left(1-\frac{1}{\ell^2}\right).
\end{equation}
\end{lemma}

We are now ready to count  the extensions of degree $p^2$ without intermediate extensions whose normal closure has a prescribed Galois group. We distinguish two cases.

\underline{Case $2\mid{}f_K$.} In this case the degree is $2$ if and only if $\max\{r,w\}=2$, and necessarily we have $s=1$. 
The group $H/\ker(\rho_0)$ is cyclic of order equal to the order $c$ of $(\alpha,\beta)$ in $\F_{p^2}^\times\times\F_{p^2}^\times$.
It follows that  $\langle {\mathcal T}_\alpha, \mathcal{U}_\beta\rangle\cong H/\ker(\rho_0)$ is the unique subgroup $N_c$ of $\F_{p^2}^\times$ of order $c$, which acts on $\F_{p^2}^+$ by multiplication, and the condition required in the proof of Lemma \ref{semidir} is verified. 

The number of elements in
$\F_{p^2}^\times\times\F_{p^2}^\times$ having order $c$ with 
$c\mid{}p^2-1$ but $c\nmid{}p-1$ s equal to $\psi(c,p^2-1)$.
On the other
hand,  for a fixed $(\alpha, \beta)$ of order $c$, the representation $J_{(\alpha,\beta)}=V_{(\alpha,\beta)}$ has multiplicity $n=[K:\Q_p]$, and, by
Proposition~\ref{CP2:submods}, the number of sub-representations $X$  which contain $J_{(\alpha, \beta)}$ is $(p^{2n}-1)/(p^2-1)$, since the
representation is defined over $\F_{p^2}$. 
Now, $X=J_{(\alpha,\beta)}\oplus J_{(\alpha^p,\beta^p)}$, hence if we sum over all pairs $(\alpha,\beta)$ of order $c$, each module $X$ is counted twice. So, we obtain that the number of classes
of extensions whose normal closure has Galois group isomorphic to $\F_{p^2}^+\rtimes{}N_c$ is exactly
\[
    \frac{p^{2n}-1}{p^2-1} \cdot \frac{1}{2}\psi(c,p^2-1).
\]
Summing over all pairs  $(\alpha,\beta)$ in $\F_{p^2}^\times\times\F_{p^2}^\times$ but
not in $\F_{p}^\times\times\F_{p}^\times$ we obtain the total number of classes of extensions with
degree $p^2$ having no intermediate extensions
\[
    \frac{p^{2n}-1}{p^2-1} \cdot \frac{1}{2}\left[ (p^2-1)^2 - (p-1)^2\right]
      = \frac{p(p^2+p-2)(p^{2n}-1)}{2(p+1)}.
\]

\underline{Case $2\nmid{}f_K$.} In this case the representations of dimension $2$ over are
obtained when one between  $r,w$ is $1$ and the other is $2$.

For $r=1$ and $w=2$ we have again $s=1$ and the group $H/\ker(\rho_0)$ is cyclic of order equal to the order $c$ of $(\alpha,\beta)$ in $\F_{p}^\times\times\F_{p^2}^\times$. 
As before, $\langle {\mathcal T}_\alpha, \mathcal{U}_\beta\rangle\cong H/\ker(\rho_0)$ is the unique subgroup of $\F_{p^2}^\times$ of order $c$, and again  the condition required in the proof of Lemma \ref{semidir} is verified. 
The pairs $(\alpha,\beta)$ of order $c$ corresponding to this case are those of the set $(\F_p^\times\times\F_{p^2}^\times)\setminus(\F_{p}^\times\times\F_{p}^\times)$. Since $c\mid{}p^2-1$
but $c\nmid{}p-1$, the possible pairs  are $\psi(c,p-1)$ and, similarly to
above, the number of extensions with Galois group $\F_{p^2}^+\rtimes{}N_c$ is
\[
    \frac{p^{2n}-1}{p^2-1} \cdot \frac{1}{2}\psi(c,p-1).
\]
Finally, the total number of extensions obtained in this way is
\[
    \frac{p^{2n}-1}{p^2-1} \cdot \frac{1}{2}\left[ (p-1)(p^2-1) - (p-1)^2\right]
      = \frac{p(p-1)(p^{2n}-1)}{2(p+1)}.
\]

\medskip

Assume now $r=2$, $w=1$ and in this case $s=2.$ The group $H$ acts on  
 $J_{(\alpha,\beta)}$ (where $\alpha\in\F_{p^2}^\times\setminus\F_p^\times$ and $\beta\in\F_p^\times$) and the action is described by the matrices
\[
   \mathcal{T}_\alpha=\begin{pmatrix} \alpha & \\ & \alpha^p \end{pmatrix},\qquad
   \mathcal{U}_\beta=\begin{pmatrix}  & \beta \\ 1 &  \end{pmatrix}.
\]
Observe that the group  $H/\ker(\rho)\cong\langle \mathcal{T}_\alpha,\mathcal{U}_\beta\rangle$ is non-abelian.
 The multiplicity of the representation is
equal to $2n$
and $J_{(\alpha,\beta)}$ is  defined over $\F_p$;  by Proposition~\ref{CP2:submods}, the number of representations isomorphic to $J_{(\alpha,\beta)}$ are 
$(p^{2n}-1)/(p-1)$.

\smallskip

We now want to classify the isomorphism classes of groups generated by $\mathcal{T}_\alpha$ and $\mathcal{U}_\beta$, as $\alpha$ and $\beta$ vary.

Let $\gamma$ be a generator of the subgroup of $\F_{p^2}^\times$ generated by $\alpha$ and $\beta$
and  put $c=\ord(\gamma)=|\langle\alpha,\beta\rangle|$; with the notation already introduced we have $\langle\mathcal{T}_\gamma\rangle\cong\langle\gamma\rangle=N_c$.

\begin{prop}\label{isocla}
The isomorphism class of the group ${\mathcal H}=\langle\mathcal{T}_\alpha,\mathcal{U}_\beta\rangle$, 
is identified by $c$ and the class of $\beta$ in $\langle\gamma\rangle/\langle\gamma\rangle^{p+1}$.
\end{prop}
\begin{proof}
Frist observe that  the matrices
$\mathcal{T}_\alpha$ and $\mathcal{U}_\beta^2=\mathcal{T}_\beta$ commute and they generate the cyclic group $\langle\mathcal{T}_\gamma\rangle$, of order $c$. This subgroup has index 2 in $ \mathcal{H}$ and it is  a maximal cyclic subgroup. It follows that $ \mathcal{H}=\langle \mathcal{T}_\gamma,\mathcal{U}_\beta\rangle$.

Multiplying $\mathcal{U}_\beta$ by a diagonal matrix in the group, we obtain
\[
 \begin{pmatrix} \gamma^i & \\ & \gamma^{ip} \end{pmatrix}
 \cdot \begin{pmatrix}  & \beta \\ 1 &  \end{pmatrix}
  = \begin{pmatrix}  & \gamma^i\beta \\ \gamma^{ip} &  \end{pmatrix},
\]
which scaling the second vector of the basis by $\gamma^{ip}$ becomes
\[
  \begin{pmatrix}  & \gamma^i\gamma^{ip}\beta \\ 1 &  \end{pmatrix} =
  \begin{pmatrix}  & \gamma^{i(p+1)}\beta \\ 1 &  \end{pmatrix}={\mathcal U}_{\gamma^{i(p+1)}\beta}
\]
(note that this scaling is a conjugation of $\mathcal{H}$ by a diagonal matrix so it fixes $\mathcal{T}_{\gamma}$). In particular, $\beta$ is defined up to elements of $\langle\gamma\rangle^{p+1}$. 
It follows that if $\langle\mathcal{T}_\alpha,\mathcal{U}_\beta\rangle$ and $\langle\mathcal{T}_{\alpha'},\mathcal{U}_{\beta'}\rangle$ are such that $|\langle\alpha,\beta\rangle|=|\langle\alpha',\beta'\rangle|$ and $\beta'=\beta\gamma^{i(p+1)}$, then the two groups are isomorphic.

\smallskip

On the other hand,  let ${\mathcal H}=\langle\mathcal{T}_\alpha,\mathcal{U}_\beta\rangle$, ${\mathcal H}'=\langle\mathcal{T}_\alpha',\mathcal{U}_\beta'\rangle$ and $\phi:\, {\mathcal H}\to{\mathcal H}'$ be an isomorphism. Then the maximal cyclic subgroups $\langle{\mathcal T}_\gamma\rangle$ and $\langle{\mathcal T}_{\gamma'}\rangle$ have the same order $c$, hence they are equal since $\langle\gamma\rangle=\langle\gamma'\rangle$; so ${\mathcal H}= \langle{\mathcal T}_\gamma, {\mathcal U}_\beta\rangle$ and ${\mathcal H}'=\langle\mathcal{T}_{\gamma},\mathcal{U}_{\beta'}\rangle$. 

\begin{lemma}\label{gamma2}
Let $\phi:\,{\mathcal H}\to{\mathcal H}'$ be an isomorphism. Then  $\phi(\langle\mathcal{T}_{\gamma^2}\rangle)=\langle\mathcal{T}_{\gamma^2}\rangle$. Moreover, one of the following holds:
\begin{enumerate}
\item[\rm{i)}] $\phi(\langle\mathcal{T}_{\gamma}\rangle)=\langle\mathcal{T}_{\gamma}\rangle$;
\item[\rm{ii)}] $\gamma^2\in\F_p^*$ and $\beta,\beta'\in\langle\gamma^2\rangle$. 
\end{enumerate} 
\end{lemma}

\begin{proof}
If $p=2$ the group $\langle\mathcal{T}_\alpha,\mathcal{U}_\beta\rangle$ is isomorphic to ${\mathcal S}_3$ and the result is trivial.

Now, let $p\ne2$. In the case when $\phi(\langle\mathcal{T}_{\gamma}\rangle)=\langle\mathcal{T}_{\gamma}\rangle$ the result is clear. Otherwise, $\phi(\mathcal{T}_{\gamma})=\mathcal{T}_{\gamma^i} {\mathcal U}_{\beta'}$ for some $i$, whence $\phi(\mathcal{T}_{\gamma^2})=\mathcal{T}_{\gamma^{i(p+1)}{\beta'}}$. Now, $\gamma^{i(p+1)},{\beta'}\in\F_p^*$ so $\gamma^2\in\F_p^*$ and $\beta,\beta'\in\langle\gamma\rangle\cap\F_p^*=\langle\gamma^2\rangle$. It follows that
$\phi(\langle\mathcal{T}_{\gamma^2}\rangle)\subseteq\langle\mathcal{T}_{\gamma^2}\rangle$ and, having the same order, they must be equal.
\end{proof}

\begin{lemma}\label{phiq=q'}
$\phi(\mathcal{T}_\beta\langle\mathcal{T}_{\gamma^{p+1}}\rangle)=\mathcal{T}_\beta' \langle\mathcal{T}_{\gamma^{p+1}}\rangle.$
\end{lemma}
\begin{proof}
Consider  the sets of squares $\mathcal{H}^2$, $\mathcal{H}'^2$ of $\mathcal{H}$ and $\mathcal{H}'$, respectively. We have $\mathcal{H}^2=\langle\mathcal{T}_{\gamma^2}\rangle\cup \mathcal{T}_\beta \langle\mathcal{T}_{\gamma^{p+1}}\rangle$ and $\mathcal{H}'^2=\langle\mathcal{T}_{\gamma^2}\rangle\cup\mathcal{T}_{\beta'} \langle\mathcal{T}_{\gamma^{p+1}}\rangle$.
Clearly $\phi(\mathcal{H}^2)=\mathcal{H}'^2$ whence, by Lemma \ref{gamma2}, $\phi(\mathcal{H}^2\setminus\langle\mathcal{T}_{\gamma^2}\rangle)=\mathcal{H}'^2\setminus\langle\mathcal{T}_{\gamma^2}\rangle.$

If $\mathcal{T}_\beta\not\in\langle\mathcal{T}_{\gamma^2}\rangle$,  then $\mathcal{H}^2\setminus\langle\mathcal{T}_{\gamma^2}\rangle=\mathcal{T}_\beta\langle\mathcal{T}_{\gamma^{p+1}}\rangle$ and 
$\phi(\mathcal{T}_\beta\langle\mathcal{T}_{\gamma^{p+1}}\rangle)=\mathcal{T}_{\beta'}\langle\mathcal{T}_{\gamma^{p+1}}\rangle$.

If $\mathcal{T}_\beta\in\langle\mathcal{T}_{\gamma^2}\rangle$, then in fact 
$\mathcal{T}_\beta \langle\mathcal{T}_{\gamma^{p+1}}\rangle\subseteq\langle\mathcal{T}_{\gamma^2}\rangle$, and the elements of 
$\mathcal{T}_\beta \langle\mathcal{T}_{\gamma^{p+1}}\rangle$ are characterised as the elements of $\mathcal{H}$ having more than 2 square roots and the same holds in $\mathcal{H'}$, so they correspond to each other under the isomorphism $\phi$.

\end{proof}

The map $\phi$ induces an automorphism $\bar\phi$ of the cyclic group $\langle\mathcal{T}_{\gamma^\epsilon}\rangle/\langle\mathcal{T}_{\gamma^{p+1}}\rangle$, where $\epsilon=1,2$ according to the cases (i) and (ii) of Lemma \ref{gamma2}. Clearly $\mathcal{T}_\beta,\mathcal{T}_{\beta'}\in\langle\mathcal{T}_{\gamma^\epsilon}\rangle$, and, by Lemma \ref{phiq=q'}, $\bar{\phi}(\overline{\mathcal{T}_\beta})=\overline{\mathcal{T}_{\beta'}}$, where $\bar{x}$  denotes the class of $x$ in $\langle\mathcal{T}_{\gamma^\epsilon}\rangle/\langle\mathcal{T}_{\gamma^{p+1}}\rangle$; in particular, $\overline{\mathcal{T}_\beta}$ and $\overline{\mathcal{T}_{\beta'}}$ have the same order. We now note that this order is  either 1 or 2 (in fact, $\beta^{p-1}=1, \beta^{p+1}\in \langle\gamma\rangle^{p+1}$, so  $\beta^2\in \langle\gamma\rangle^{p+1}$, i.e., $\overline{\mathcal{T}_{\beta}}^2=\bar1$).  Since in the cyclic group $\langle\mathcal{T}_{\gamma^\epsilon}\rangle/\langle\mathcal{T}_{\gamma^{p+1}}\rangle$ there is at most an element of order 2 and an element of order 1, then $\overline{\mathcal{T}_\beta}=\overline{\mathcal{T}_{\beta'}}$, so $\beta\langle\gamma\rangle^{p+1}=\beta'\langle\gamma\rangle^{p+1}$ thus proving the proposition.
\end{proof}

We now show that each group $\mathcal{H}= \langle{\mathcal T}_\gamma, {\mathcal U}_\beta\rangle$ is conjugate (via a diagonal matrix) to a subgroup of ${\mathcal H}_0=\langle\mathcal{T}_{\gamma_0},\mathcal{U}_1\rangle$, where $\gamma_0$ is a generator of $\F_{p^2}^\times$. 
Observe that ${\mathcal H}_0=\{T_{\gamma_0^i}, V_i\ |\ i=0,\dots,p^2-1\}$, where 
$$
V_i =\begin{pmatrix} &\gamma_0^i \\  \gamma_0^{ip}& \end{pmatrix}.
$$
Let  $m$ and $b$ be integers such that $\gamma=\gamma_0^m$ and $\beta=\gamma^b=\gamma_0^{mb}$. Since $\beta\in\F_p^\times$, then $p+1|mb$ and we may write $\beta=\gamma_0^{(p+1)j}$ where $j\equiv\frac{mb}{p+1}\pmod{p-1}.$ Let also 
$$
M_j= \begin{pmatrix} 1& \\ & \gamma_0^{jp} \end{pmatrix}.
$$
It is easy to check that $M_j{\mathcal T}_\gamma M_j^{-1}={\mathcal T}_\gamma$ and $M_j{\mathcal U}_\beta M_j^{-1}=V_j$. It follows that the group ${\mathcal H}$ is conjugate to the subgroup $\langle {\mathcal T}_\gamma, V_j\rangle=\{T_{\gamma^i},
 V_{j+mi}\ |\ i=0,\dots, \frac{p^2-1}{(m,p^2-1)}-1\}$  of ${\mathcal H}_0$; we note that this subgroup depends only on the class of $j$ modulo $m$, and since $j$ is determined modulo $p-1$, the subgroup is determined by the class of $j$ modulo $(m,p-1)$. 

Now, if $\mathcal{H}'$ is isomorphic to $\mathcal{H}$ then, by Proposition \ref{isocla}, $\mathcal{H}' = \langle{\mathcal T}_\gamma, {\mathcal U}_{\beta'}\rangle$ with $\beta'\langle\gamma^{p+1}\rangle=\beta\langle\gamma^{p+1}\rangle$, namely 
$$\frac{mb'}{p+1}\equiv\frac{mb}{p+1}\pmod{(m,p-1)},$$
hence  the subgroups of $\mathcal{H}_0$ conjugate to $\mathcal{H}$ and $\mathcal{H}'$ whit this construction are the same. In particular, $\mathcal{H}$ and $\mathcal{H}'$ are conjugate subgroups of $GL_2(\F_{p^2}),$
as required in the proof of  Lemma \ref{semidir}.

We are now ready to count the isomorphism classes of extensions in the case $2\nmid f_K$ with a fixed Galois group. 

Our first step will be to count the pairs $(\alpha, \beta)$ for which $c={\rm ord}(\alpha,\beta)$ and the class of $\beta$ modulo $\langle\gamma^{p+1}\rangle$ are fixed.
The number of pairs $(\alpha, \beta)$ of order $c$ is $\psi(c,p-1)$ and we must distinguish these pairs according to the class of $\beta$. 

\begin{lemma}
For $c\in\NN$ with $c|p^2-1$ and $c\nmid p-1$ let
$$
   \lambda(c,p-1) = \begin{cases}
     1  & \text{if }v_2(c)=0\text{ or }v_2(c)=v_2(p^2-1),\\
    1/2 & \text{if }v_2(p-1)<v_2(c)<v_2(p^2-1),\\
    1/3 & \text{if }0<v_2(c)\leq{}v_2(p-1).
   \end{cases}
$$
Then the number of pairs $(\alpha,\beta)\in\F_{p^2}^\times\times\F_p^\times$ of  order $c$ such that $\beta\in\langle\gamma^{p+1}\rangle$ is $\lambda(c,p-1)\psi(c,p-1)$, while
 the number of pairs such that $\beta\notin\langle\gamma^{p+1}\rangle$ is $(1-\lambda(c,p-1))\psi(c,p-1)$.
\end{lemma}
\begin{proof}
Let  ${\mathcal H}=\langle\mathcal{T}_{\alpha},\mathcal{U}_\beta\rangle=\langle\mathcal{T}_{\gamma},\mathcal{U}_\beta\rangle$  with ${\rm ord}(\gamma)=c$.
 Since  $(p-1,p+1)=2$ (or $1$ for $p=2$) we have that
$(\langle\gamma\rangle\cap{}\F_p^\times)/\langle\gamma^{p+1}\rangle$ has order $1$ or $2$, so it is the quotient of  the 2-Sylow subgroups of $\langle\gamma\rangle\cap{}\F_p^\times$ and $\langle\gamma^{p+1}\rangle$. It is easy to check that the order of this quotient is 2 when $0<v_2(c)<v_2(p^2-1)$ and  it is 1 otherwise, namely if $v_2(c)=0$ or $v_2(c)=v_2(p^2-1)$.

When this order is 1, necessarily $\beta\in \langle\gamma^{p+1}\rangle$. 

Consider now the case when the order is 2 (in this case necessarily $p\ne2$): we have that $\beta\in\langle\gamma^{p+1}\rangle$ if and only if $v_2(\ord(\beta))\le v_2(\ord(\gamma^{p+1}))=v_2\left(\frac{c}{(c,p+1)}\right)$, so
the condition $\beta \in\langle\gamma^{p+1}\rangle$ depends only on the 2-component $(x,y)$ of $(\alpha,\beta)$ in   the decomposition of  $\F_{p^2}^\times\times\F_p^\times$ as a direct sum of its $\ell$-Sylow subgroups.

The 2-Sylow of $\F_{p^2}^\times\times\F_p^\times$ is isomorphic to  $\Z/2^{w+z}\Z\times{}\Z/2^z\Z$, where $2^w\|(p+1)$ and $2^z\|(p-1)$. Note that either $z=1$ or $w=1$. Assume
$2^k\|c$ for some $1\leq{}k<w+z$. If $z=1$ and $k=1$, the possible pairs $(x,y)$  are $(1,-1),(-1,-1),(-1,1)$ and the only one with $\beta\in \langle\gamma^{p+1}\rangle$ is $(-1,1),$ giving $\frac13$ of the cases.
If $z=1$ and $k>1$ the possible pairs $(x,y)$ are $(x,\pm1)$ where $\ord(x)=2^k$, and $\beta\in \langle\gamma^{p+1}\rangle$ if and only if $y=1$, giving $\frac12$ of the cases. 
If $w=1$, then $\beta\in \langle\gamma^{p+1}\rangle$ precisely when $(x,y)$ is such that $x$ has order bigger than $y$, and it is easy to verify that this happens $1/3$ of the
times. With our definition of  $\lambda(c,p-1)$ the lemma follows.
\end{proof}

\begin{remark}
{\rm A perhaps more intrinsic characterization of the property $\beta\in\langle\gamma^{p+1}\rangle$ is that it holds if and only if the sequence
$1\rightarrow\langle\gamma\rangle\rightarrow{}{\mathcal H}\rightarrow\frac{\mathcal H}{\langle\gamma\rangle}\rightarrow1$ splits.
}
\end{remark}
To count the isomorphism classes we have to take into account that the pairs $(\alpha,\beta)$ and $(\alpha^p,\beta)$ give the same representation, so the number of  pairs just counted must be divided by 2.  By Proposition \ref{CP2:submods} multiplying
by $(p^{2n}-1)/(p-1)$ we obtain the number of isomorphism classes of extensions having a
particular group.

The total number of classes of extensions for $r=2$, $w=1$ is obtained as
\[
   \frac{p^{2n}-1}{p-1} \frac{1}{2}\left[ (p-1)(p^2-1) - (p-1)^2 \right]
    = \frac{1}{2} p(p-1)(p^{2n}-1),
\]
and the total number of classes of extension having no intermediate extension is again
\[
   \frac{1}{2} p(p-1)(p^{2n}-1) + \frac{p(p-1)(p^{2n}-1)}{2(p+1)} = \frac{p(p^2+p-2)(p^{2n}-1)}{2(p+1)}.
\]

We collect all the results obtained in the following theorem.

\begin{theo}
  Let $K$ be an extension of $\Q_p$ of degree $n$. Let $c$ be an integer dividing $(p^2-1)$ but not
  $(p-1)$, and let $N_c$ be the cyclic subgroup of $\F_{p^2}^\times$ of order $c$. Let $\cG(N_c)$ be the
  number of isomorphism classes of extensions of degree $p^2$ such that the normal closure has group
  isomorphic to $\F_{p^2}^+\rtimes{N_c}$. Then
  \[
    \cG(N_c) = \frac{p^{2n}-1}{p^2-1} \cdot \frac{1}{2}\psi(c,p^{(f_k,2)}-1).
  \]
  Let ${\mathcal H}$ be a non-abelian subgroup of $\F_{p^2}^\times\rtimes\Gal(\F_{p^2}/\F_p)$ not contained in
  $\F_{p^2}^\times$. Let $\cG({\mathcal H})$ be the number of isomorphism classes of extensions of degree $p^2$
  such that the normal closure has Galois group isomorphic to $\F_{p^2}^+\rtimes{}{\mathcal H}$. Put
  $N_c={\mathcal H}\cap{}\F_{p^2}^\times$ and $c=|N_c|$. If $2\mid{}f_K$ then $\cG({\mathcal H})=0$, while if $2\nmid{}f_K$ we
  have
  \[
    \cG({\mathcal H}) = \left\{\begin{array}{cl}
      \lambda(c,p-1)\cdot \frac{p^{2n}-1}{2(p-1)} \cdot \psi(c,p-1) 
       & \text{if $N_c\rightarrow{}{\mathcal H}$ splits,} \\
      \left(1-\lambda(c,p-1)\right) \cdot \frac{p^{2n}-1}{2(p-1)} \cdot \psi(c,p-1) &
       \text{if $N_c\rightarrow{}{\mathcal H}$ does not split.}
    \end{array}\right.
  \]
  The above groups exhaust the Galois groups of normal closures of isomorphism classes of extensions of
  degree $p^2$ having no intermediate extension. The total number $\cK_K$ of isomorphism classes of
  extensions of degree $p^2$ with no intermediate extension is
  \[
    \cK_K = \frac{p(p^2+p-2)(p^{2n}-1)}{2(p+1)}.
    \]
\end{theo}

In accordance with the data in \cite{jones2006database}, we obtain $4$ classes of extensions of degree $4$ over $\Q_2$ with no intermediate extension, and $30$ of degree $9$ over $\Q_3$.
Finally, we note that each isomorphism class $[L]$ contains exactly $p^2$ extensions; in fact,  the subgroup of $\Gal(\tilde L/K)$ fixing $L$ coincides with its normalizer (because $L/K$ has no intermediate extension) and therefore this normaliser has index $p^2$. It follows that the total number of extensions of  degree $p^2$ with no intermediate extension is
$$ \frac{p^3(p^2+p-2)(p^{2n}-1)}{2(p+1)}.$$

\bibliographystyle{amsalpha_abbr}
\bibliography{biblio}

\end{document}